\documentclass[a4paper,leqno,10pt]{amsart}

\usepackage{epsf,graphicx,epsfig}
\usepackage{amscd}
\usepackage{amsmath,latexsym,amssymb,amsthm}
\usepackage[nospace,noadjust]{cite}
\usepackage{textcomp}
\usepackage{setspace,cite}
\usepackage{lscape,fancyhdr,fancybox}
\usepackage{textcomp}
\usepackage{bm}
\usepackage{stmaryrd}
\usepackage[all,cmtip]{xy}
\usepackage{tikz}
\usetikzlibrary{shapes,arrows,decorations.markings}
\usepackage{graphicx}

\usepackage{color}
\usepackage{url}
\usepackage{enumerate}
\usepackage[mathscr]{euscript}

  \theoremstyle{plain}

\swapnumbers
    \newtheorem{thm}{Theorem}[section]
    \newtheorem{prop}[thm]{Proposition}
   \newtheorem{lemma}[thm]{Lemma}

    \newtheorem{subsec}[thm]{}

\theoremstyle{definition}
    \newtheorem{defn}[thm]{Definition}
    \newtheorem{remark}[thm]{Remark}
    \newtheorem{exam}[thm]{Example}

\theoremstyle{remark}

\title{}
\author{}
\date{}
\usepackage{amssymb}

\usepackage{hyperref}
\hypersetup{
	colorlinks,
	citecolor=blue,
	filecolor=black,
	linkcolor=blue,
	urlcolor=black
}

\title{Contact Courant algebroids and $L_\infty$-algebras}
\author{Apurba Das}
\email{apurbadas348@gmail.com}
\address{Department of Mathematics and Statistics,
Indian Institute of Technology, Kanpur 208016, Uttar Pradesh, India}

\begin{document}

\begin{abstract}
Let $L$ be a line bundle over $M$. In this paper we associate an $L_\infty$-algebra to any $L$-Courant algebroid (contact Courant algebroid in the sense of Grabowski). This construction is similar to the work of Roytenberg and Weinstein for Courant algebroids. Next we associate a $p$-term $L_\infty$-algebra to any isotropic involutive subbundle of $(\mathbb{D}L)^p := DL \oplus (\wedge^p (DL)^* \otimes L)$, where $DL$ is the gauge algebroid of $L$. In a particular case, we relate these $L_\infty$-algebras by a suitable morphism.
\end{abstract}
\maketitle
\thispagestyle{empty}


\section{Introduction}
The notion of Courant algebroid was introduced by Liu, Weinstein and Xu as a double object of a Lie bialgebroid \cite{liu-wein-xu}. Courant algebroids also appear in multisymplectic geometry, topological field theory \cite{rogers1, roy3}. In \cite{roy-wein}, Roytenberg and Weinstein showed that Courant algebroid gives rise to an $L_\infty$-algebra. Later, a truncated $L_\infty$-algebra associated to a Courant algebroid has been considered by Rogers \cite{rogers1}.

The contact analogue of Courant algebroids was introduced by Grabowski under the name of contact Courant algebroids \cite{grab}. This is similar to Courant algebroids where the tangent bundle of the base manifold $M$ is replaced by the gauge algebroid of a line bundle $L$ over $M$, the $C^\infty(M)$-valued pairing is replaced by the $L$-valued pairing. To make the underlying line bundle $L$ explicit, we use the terminology $L$-Courant algebroid (instead of contact Courant algebroid). However, the original definition of contact Courant algebroid is slightly different than ours.
 A similar notion of $E$-Courant algebroid has been introduced in \cite{chen-liu-sheng}. The model example of an $L$-Courant algebroid is given by $(\mathbb{D}L) := DL \oplus ((DL)^* \otimes L)$ with the structure similar to the standard Courant algebroid, where $DL$ is the gauge algebroid of $L$. One may also twist this structure by a closed Atiyah $3$-form $\omega \in \Omega^3_{cl} (DL,L)$ of the Atiyah complex of $L$. We denoted this $L$-Courant algebroid by $(\mathbb{D}L)_\omega$.

Like Roytenberg and Weinstein, we associate an $L_\infty$-algebra to any $L$-Courant algebroid (cf. Theorem \ref{roy-wein-l}). However, in this paper, we are interested in a truncated form of this $L_\infty$-algebra (cf. Theorem \ref{new-2-term}). We also find some interesting results regarding this $L_\infty$-algebra. The $L_\infty$-algebra corresponding to the $L$-Courant algebroid $\mathbb{D}L$ can be seen as a semidirect product $L_\infty$-algebra of a representation up to homotopy (cf. Proposition \ref{semi-direct}). Moreover, there is a morphism from the Lie algebra $\Gamma(DL)$ of derivations to the $L_\infty$-algebra associated to the $L$-Courant algebroid $\mathbb{D}L$ (cf. Theorem \ref{morphism-thm}). Finally, we observe that the $2$-term $L_\infty$-algebras induced from $L$-Courant algebroids defined by cohomologous Atiyah $3$-forms are isomorphic (cf. Proposition \ref{cohomologous-thm}).

In the next, we consider exact $L$-Courant algebroids and associate a third Atiyah cohomology class of $L$ to any exact $L$-Courant algebroid. This approach is similar to the work of \v{S}evera \cite{severa} for exact Courant algebroids. However, due to the acyclicity of the Atiyah complex, the corresponding third Atiyah cohomology class is zero.

The notion of $\mathcal{E}^1(M)$-Dirac structure was first introduced by Wade \cite{wade} (see also \cite{costa}). Later on, this notion has been formalized under the name of Dirac-Jacobi structures (or, Dirac-Jacobi bundles) \cite{vitag2}. Given a line bundle $L$ over $M$, a Dirac-Jacobi structure on $L$ is a maximally isotropic involutive subbundle of $(\mathbb{D}L): = DL \oplus ((DL)^* \otimes L)$. Jacobi structures on a line bundle $L$ can be seen as Dirac-Jacobi structures on $L$ \cite{vitag2}. We also define $\omega$-twisted Jacobi structures on a line bundle $L$, where $\omega \in \Omega^3_{cl} (DL,L)$ is a closed Atiyah $3$-form. An $\omega$-twisted Jacobi structure on $L$ induces a Lie algebroid structure on the $1$-jet bundle $J^1L$ (cf. Proposition \ref{twisted-j-lie}). Motivated from the gauge transformations of Poisson structures \cite{severa-wein}, we also define gauge transformations of Jacobi structures by a closed Atiyah $2$-form.

In the next, we show our interest on isotropic involutive subbundles of 
 $(\mathbb{D}L)^p := DL \oplus (\wedge^p (DL)^* \otimes L)$. The graph of a closed Atiyah $(p+1)$-form $\omega \in \Omega^{p+1}_{\text{cl}} (DL,L)$ defined by
$$\text{Gr}(\omega) = \{ (\triangle, i_\triangle \omega) |~ \triangle \in DL \} \subset (\mathbb{D}L)^p $$
is an isotropic and involutive subbundle. An isotropic involutive subbundle $\xi \subset  (\mathbb{D}L)^p$ inherits a Lie algebroid structure.

In the next, we associate an $L_\infty$-algebra to any isotropic involutive subbundle $\xi \subset  (\mathbb{D}L)^p$. This is similar to Rogers for multisymplectic structure and Zambon for higher Dirac structure \cite{rogers, zambon}. An Atiyah $(p-1)$-form $\alpha \in \Omega^{p-1}(DL,L)$ is called Hamiltonian if there exists a derivation $\triangle_\alpha \in \Gamma (DL)$ such that $(\triangle_\alpha, d_{DL} \alpha) \in \Gamma \xi$. The set of Hamiltonian Atiyah forms are denoted by $\Omega^{p-1}_{\text{Ham}}(DL,L)$. There is a bracket $\{-,-\}$ on $\Omega^{p-1} (DL,L)$ defined by
$$\{\alpha, \beta \} := i_{\triangle_\alpha} d_{DL} \beta,$$
where $\triangle_\alpha$ is any Hamiltonian derivation associated to $\alpha$. This bracket is skew-symmetric but need not satisfy the Jacobi identity in general. However, we show that there is a $p$-term $L_\infty$-algebra on the complex
\begin{align}\label{first-eqn}
\Gamma L \xrightarrow{d_{DL}} \Omega^1(DL,L) \xrightarrow{d_{DL}} \Omega^2(DL,L) \xrightarrow{d_{DL}} \cdots \cdot \xrightarrow{d_{DL}} \Omega^{p-2} (DL,L) \xrightarrow{d_{DL}} \Omega^{p-1}_{Ham}(DL,L)
\end{align}
(Theorem \ref{l-inf-iso-inv}). This $L_\infty$-algebra is called the algebra of observables of the isotropic involutive subbundle $\xi$. For any closed non-degenerate Atiyah $3$-form $\omega \in \Omega^3_{\text{cl}} (DL,L)$, we show that there is an injective morphism of $L_\infty$-algebras from the one induced from the isotropic involutive subbundle Gr$(\omega)$ to the one induced from the $L$-Courant algebroid $(\mathbb{D}L)_\omega$ (Theorem \ref{inj-morphism}). Finally, we also associate a dg (differential graded) Leibniz algebra structure on the complex (\ref{first-eqn}) associated to the isotropic involutive subbundle Gr$(\omega)$ induced from the closed non-degenerate Atiyah $(p+1)$-form $\omega$  (cf. Theorem \ref{dg-leibniz-thm}).

\medskip

\section{Preliminaries}
In this section we recall some basic definitions which are essential to understand the main contents of the paper \cite{lada-markl, abad-crainic, chen-liu}.

\subsection{$L_\infty$-algebras}
\begin{defn}
An $L_\infty$-algebra consists of a graded vector space $\mathcal{A} = \oplus \mathcal{A}_i$ together with a collection $\{ l_k |~ 1 \leq k < \infty \}$ of multilinear maps $l_k : \mathcal{A}^{\otimes k} \rightarrow \mathcal{A}$ with $\text{deg} (l_k) = k-2$, satisfying
\begin{itemize}
\item (skew-symmetry) for all $k \geq 1$,
\begin{center}
$l_k (a_{\sigma (1)}, \ldots, a_{\sigma (k)}) = (-1)^\sigma \epsilon (\sigma)~ l_k (a_1, \ldots, a_k)$,
\end{center}

\item (higher Jacobi identity) for all $n \geq 1$,
\begin{align}\label{hl1-iden}
	\sum_{i+j = n+1}^{} \sum_{\sigma}^{} \text{sgn}(\sigma)~ \epsilon (\sigma) ~ (-1)^{i (j-1)} ~ l_j   \big(  l_i (a_{\sigma (1)}, \ldots, a_{\sigma (i)}), a_{\sigma (i+1)}, \ldots, a_{\sigma (n)} \big) = 0,
	\end{align}
	\end{itemize}
	for all $a_1, \ldots, a_n \in \mathcal{A}$, and $\sigma$ runs over all $(i, n-i)$ unshuffles with $i \geq 1$. The notation $\epsilon (\sigma)$ is the usual Koszul sign 
in the graded context, and $\text{sgn} (\sigma)$ denote the signature of the permutation $\sigma$.
\end{defn}

It follows from the above definition that the degree $-1$ map $l_1 : \mathcal{A} \rightarrow \mathcal{A}$ is a differential. In other words, $(\mathcal{A}, l_1)$ is a chain complex. Moreover, $l_1$ satisfies the following Leibniz rule for the degree $0$ skew-symmetric product $l_2 : \mathcal{A}^{\otimes 2} \rightarrow \mathcal{A}$:
$$l_1 (l_2 (a,b)) = l_2 (l_1 (a), b) + (-1)^{|a|}~ l_2 (a, l_1 (b)),$$
for all $a, b \in \mathcal{A}$. However, the product $l_2$ need not satisfy the graded Jacobi identity, but it does up to some terms involving $l_3$. Similarly, for higher $n$, we get higher coherence laws that $l_k$'s must satisfy.

An $n$-term $L_\infty$-algebra is an $L_\infty$-algebra $(\mathcal{A}, l_k)$ whose underlying graded vector space $\mathcal{A}$ is concentrated in degrees $0, 1, \ldots, n-1$. The corresponding chain complex $(\mathcal{A}, l_1)$ is given by
\begin{align*}
\mathcal{A}_{n-1} \xrightarrow{l_1} \mathcal{A}_{n-2} \xrightarrow{l_1} \cdots  \xrightarrow{l_1} \mathcal{A}_{1} \xrightarrow{l_1} \mathcal{A}_{0}.
\end{align*}
In this case, it is easy to see that $l_k = 0$, for $k > n+1$.

\begin{defn}\label{l-inf-map}
Let $\mathcal{A}=(\mathcal{A}_1 \xrightarrow{l_1} \mathcal{A}_0, l_2, l_3)$ and $\mathcal{A}' =(\mathcal{A}'_1 \xrightarrow{l'_1} \mathcal{A}'_0, l'_2, l'_3)$ be two $2$-term $L_\infty$-algebras. A morphism $\phi : \mathcal{A} \rightsquigarrow \mathcal{A}'$ of $2$-term $L_\infty$-algebras consist of a chain map $\phi: \mathcal{A} \rightarrow \mathcal{A}'$ (which consists of maps maps $\phi_0 : \mathcal{A}_0 \rightarrow \mathcal{A}'_0$ and $\phi_1 : \mathcal{A}_1 \rightarrow \mathcal{A}'_1$ satisfying $\phi_0 \circ l_1 = l_1' \circ \phi_1$) and a skew-symmetric map $\phi_2 : \mathcal{A}_0 \times \mathcal{A}_0 \rightarrow \mathcal{A}_1'$ satisfying

\begin{itemize}
\item $ l'_2 (\phi_0 (x), \phi_0 (y)) - \phi_0 l_2 (x,y) = l'_1 (\phi_2 (x,y)),$
\item $ l_2' (\phi_0 (x), \phi_1 (h)) - \phi_1 l_2 (x,h) = \phi_2 (x, l_1 h),$
\item $ l_3' (\phi_0 (x), \phi_0 (y), \phi_0 (z)) - \phi_1 (l_3 (x,y,z)) = \phi_2 (x, l_2 (y,z)) + \phi_2 (y, l_2 (z,x)) + \phi_2 (z, l_2 (x,y)) +  l_2' (\phi_0 (x), \phi_2 (y,z)) + l_2' (\phi_0 (y), \phi_2 (z,x)) + l_2' (\phi_0 (z), \phi_2 (x,y)),$
\end{itemize}
for all $x, y, z \in \mathcal{A}_0$ and $h \in \mathcal{A}_1$.
\end{defn}

A strict morphism is a morphism $\phi$ in which $\phi_2 = 0$. A strict isomorphism is a strict morphism in which $\phi_0$ and $\phi_1$ are isomorphisms.

\subsection{Representation up to homotopy of a Lie algebra}
Here we recall the definition of a representation up to homotopy of a Lie algebra \cite{abad-crainic} (viewed as a Lie algebroid over a point). Although, we are only interested in $2$-terms representation up to homotopy.
Let $(\mathfrak{g}, [-,-])$ be a Lie algebra.

\begin{defn}
A $2$-term representation up to homotopy of $\mathfrak{g}$ consists of a chain complex $V_1 \xrightarrow{d} V_0$ together with
\begin{itemize}
\item two linear maps $\mu_i : \mathfrak{g} \rightarrow End (V_i)$ satisfying 
\begin{center}
$ d ( \mu_1 (X) (s)) = \mu_0 (X) (d s),$
\end{center}
 \item an element $ \nu \in \Omega^2 ( \mathfrak{g} , Hom (V_0, V_1))$ such that
 \begin{itemize}
\item[(i)] $\mu_0 [X,Y] - [\mu_0 (X), \mu_0 (Y)] ~=~ d \circ \nu (X,Y),$
\item[(ii)] $\mu_1 [X,Y] - [\mu_1 (X), \mu_1 (Y)] ~=~ \nu (X,Y) \circ d,$
\item[(iii)] $d \nu = 0,$
\end{itemize}
\end{itemize}
for $X, Y \in \mathfrak{g}$ and $s \in V_1$.
\end{defn}

We denote a representation up to homotopy of $\mathfrak{g}$ by $(V_1 \xrightarrow{d} V_0, \mu_0, \mu_1, \nu).$ Given a representation up to homotopy $(V_1 \xrightarrow{d} V_0, \mu_0, \mu_1, \nu)$, one can form a new $2$-term chain complex
\begin{align}\label{new-2-term-complex}
V_1 \xrightarrow{d} (\mathfrak{g} \oplus V_0).
\end{align} 

In the case of a $2$-term representation up to homotopy, we have the following \cite{abad-crainic}.

\begin{prop}\label{thm-new-2-term}
Let $(V_1 \xrightarrow{d} V_0, \mu_0, \mu_1, \nu)$ be a $2$-term representation up to homotopy of $\mathfrak{g}$. Then the complex (\ref{new-2-term-complex}) inherits a $2$-term $L_\infty$-algebra whose structure maps are given by
\begin{align*}
l_2 (X + \alpha, Y + \beta) =~& [X,Y] + \mu_0 (X)(\beta) - \mu_0 (Y)(\alpha),\\
l_2 (X+ \alpha, s) =~& \mu_1 (X) (s),\\
l_2 (s, X + \alpha) =~& - \mu_1 (X) (s),\\
l_3 (X+ \alpha, Y+ \beta, Z + \gamma) =~& - \nu (X,Y)(\gamma) + c.p.,
\end{align*}
for $X+ \alpha,~ Y+ \beta,~ Z + \gamma \in \mathfrak{g} \oplus V_0$; $s \in V_1$ and
where $c.p.$ stands for cyclic permutation.
\end{prop}

The above $L_\infty$-algebra is called the semidirect product of the representation up to homotopy.

\subsection{Gauge algebroids}
Let $L$ be a vector bundle over $M$. A derivation on $L$ is an $\mathbb{R}$-linear map $\triangle : \Gamma L \rightarrow \Gamma L$ satisfying
$$\triangle (fs) = f \triangle (s) + X(f) s,~~ \text{ for } f \in C^\infty(M), s \in \Gamma L,$$
for a necessarily unique vector field $X \in \Gamma (TM)$. The vector field $X$ is called the symbol of $\triangle$ and is denoted by $\sigma (\triangle)$. Derivations are sections of a Lie algebroid $DL \rightarrow M$, called the gauge algebroid of $L$. The Lie bracket $[-,-]$ is given by the commutator of derivations and the anchor is given by the symbol map. Note that, the Lie algebroid $DL$ has a tautological representation on $L$ given by the action of a derivation on a section. The corresponding Lie algebroid cohomology complex is given by
$$ \Omega^0 (DL,L) \xrightarrow{d_{DL}} \Omega^1(DL,L) \xrightarrow{d_{DL}} \cdots \cdot \xrightarrow{d_{DL}} \Omega^n(DL,L) \xrightarrow{d_{DL}} \Omega^{n+1}(DL,L) \xrightarrow{d_{DL}} \cdots \cdot,$$
where $\Omega^k (DL,L) = \Gamma (\wedge^{k} (DL)^* \otimes L)$, for $k \geq 0$. This is called the Atiyah complex of $L$ and the elements of $\Omega^\bullet (DL,L)$ are called Atiyah forms on $L$. The Atiyah complex of $L$ is acyclic. Since $L$ is a line bundle, there is a vector bundle isomorphism $J^1L \simeq (DL)^* \otimes L$, where $J^1L$ is the first jet bundle of $L$. Then the Lie algebroid differential $\Gamma L = \Omega^0(DL,L) \xrightarrow{d_{DL}} \Omega^1(DL,L) = \Gamma(J^1L)$ is just the jet prolongation.

\medskip

\section{$L$-Courant algebroids}
The notion of contact Courant algebroids was introduced by Grabowski as a contact analogue of Courant algebroids \cite{grab}. To make the underlying line bundle $L$ explicit, we use the terminology $L$-Courant algebroid (instead of contact Courant algebroid). However, the original definition of contact Courant algebroid is slightly different than ours.

Let $L$ be a line bundle over $M$.
\begin{defn}
An $L$-Courant algebroid consists of a vector bundle $E \rightarrow M$ together with
\begin{itemize}
\item a bracket $[-,-] : \Gamma E \times \Gamma E \rightarrow \Gamma E,$
\item a symmetric, non-degenerate $L$-valued pairing $\langle -,- \rangle : \Gamma E \times \Gamma E \rightarrow \Gamma L,$
\item a bundle map $\rho : E \rightarrow DL$
\end{itemize}
satisfying
\begin{itemize}
\item[(LC1)] $~[e_1, [e_2,e_3]] = [[e_1, e_2], e_3] + [e_2, [e_1, e_3]],$
\item[(LC2)] $~[e_1, f e_2 ] = f [e_1, e_2] + \sigma (\rho (e_1))(f) e_2,$
\item[(LC3)] $~ \rho [e_1, e_2] = [\rho (e_1), \rho (e_2) ],$
\item[(LC4)] $~ [e , e] = \frac{1}{2} \mathcal{D} \langle e, e \rangle,$
\item[(LC5)] $~ \rho (e) \langle e_1, e_2 \rangle = \langle [e, e_1], e_2 \rangle + \langle e_1, [e, e_2] \rangle,$
\end{itemize}
for all $e, e_1, e_2, e_3 \in \Gamma E$; $f \in C^\infty(M)$, where $\mathcal{D} : \Gamma L \rightarrow \Gamma E$ is given by the following compositions
$$ \Gamma L = \Omega^0 (DL,L) \xrightarrow{d_{DL}} \Omega^1(DL,L) \xrightarrow{\rho^*} \Gamma (E^* \otimes L) \xrightarrow{\sim} \Gamma E.$$
\end{defn}
The last identification is induced from the pairing $\langle -, - \rangle$. An $L$-Courant algebroid is denoted by $(E, [-,-], \langle -, - \rangle, \rho)$.
A similar notion of $E$-Courant algebroid has been defined in \cite{chen-liu-sheng} where the underlying bundle $E$ may not be a line bundle.

Let $(E, [-,-], \langle -, - \rangle, \rho)$ be an $L$-Courant algebroid. Define a new bracket $\llbracket -,- \rrbracket : \Gamma E \times \Gamma E \rightarrow \Gamma E$ by
\begin{align}\label{skew-transform}
\llbracket e_1, e_2 \rrbracket = [e_1, e_2] - \frac{1}{2} \mathcal{D} \langle e_1, e_2 \rangle .
\end{align}
It follows from (LC4) that $\llbracket e, e \rrbracket = 0$. Hence the bracket $\llbracket -, - \rrbracket$ is skew-symmetric.

\begin{exam}(Generalized Courant algebroids / Courant-Jacobi algebroids) A generalized Courant algebroid in the sense of \cite{costa} is a vector bundle $E \rightarrow M$ together with a skew-symmetric bracket $\llbracket -,- \rrbracket$ on the space of sections of $E$, a non-degenerate symmetric pairing $\langle -,- \rangle : \Gamma E \times \Gamma E \rightarrow C^\infty(M)$ and a bundle map $\rho^\theta : E \rightarrow TM \times \mathbb{R}$ satisfying some properties similar to the skew-symmetric version of the Courant algebroid. A non skew-symmetric analogue of generalized Courant algebroid is an $L$-Courant algebroid for the trivial line bundle $L$.

\end{exam}

\begin{exam}\label{omni}
 Let $L$ be a line bundle over $M$. Then $$\mathbb{D}L := DL \oplus J^1L$$
is an $L$-Courant algebroid whose bracket, $L$-valued pairing and anchor are given by
\begin{align*}
[(\triangle, \alpha), (\nabla, \beta)] =~& ([\triangle, \nabla],~ \mathcal{L}_\triangle \beta - i_\nabla d_{DL} \alpha),\\
\langle (\triangle, \alpha), (\nabla, \beta) \rangle =~& \alpha (\nabla) + \beta (\triangle),\\
\rho =~& pr_1,
\end{align*}
for $(\triangle, \alpha), (\nabla, \beta) \in \Gamma (\mathbb{D}L),$ where $pr_1$ denotes the projection onto the first factor. The algebraic structure on $\mathbb{D}L := DL \oplus J^1L$ is also known as omni-Lie algebroid in the literature \cite{chen-liu}. 
\end{exam}

\begin{exam}\label{twisted-omni}
Let $\omega \in \Omega^3 (DL,L)$ be a closed Atiyah $3$-form on $L$. Then one can twist the above bracket on $\Gamma (\mathbb{D}L)$ by $\omega$,
\begin{align*}
[(\triangle, \alpha), (\nabla, \beta)]_\omega =~& ([\triangle, \nabla],~ \mathcal{L}_\triangle \beta - i_\nabla d_{DL} \alpha - i_\nabla i_\triangle \omega).
\end{align*}
This bracket together with the above pairing and anchor forms a new (twisted) $L$-Courant algebroid structure on $\mathbb{D}L$. We denote this $L$-Courant algebroid by $(\mathbb{D}L)_\omega$. The corresponding skew-symmetric bracket is given by
\begin{align*}
\llbracket (\triangle, \alpha), (\nabla, \beta) \rrbracket_\omega = \big( [\triangle, \nabla],~ \mathcal{L}_\triangle \beta - \mathcal{L}_\nabla \alpha - \frac{1}{2} d_{DL} (\beta (\triangle) - \alpha (\nabla)) - i_\nabla i_\triangle \omega \big).
\end{align*}
\end{exam}


It is known from the work of Roytenberg and Weinstein \cite{roy-wein} that Courant algebroid gives rise to an $L_\infty$-algebra structure. In a similar way, any $L$-Courant algebroid gives an $L_\infty$-algebra.

Let $(E, [-,-], \langle -, - \rangle, \rho)$ be an $L$-Courant algebroid with the corresponding skew-symmetric bracket $\llbracket -, - \rrbracket$ given by (\ref{skew-transform}).
For any $e_1, e_2, e_3 \in \Gamma E$, we define
\begin{align*}
T (e_1, e_2, e_3) = \frac{1}{6} \langle \llbracket e_1, e_2 \rrbracket, e_3 \rangle + c.p.
\end{align*}

\begin{thm}\label{roy-wein-l}
Let $(E, [-,-], \langle -,- \rangle, \rho)$ be an $L$-Courant algebroid. Then the chain complex 
$$ \text{ker } \mathcal{D} \xrightarrow{i} \Gamma L \xrightarrow{\mathcal{D}} \Gamma E$$
carries a $3$-term $L_\infty$-algebra with the structure maps
\begin{align*}
l_2 (e_1 \wedge e_2) =~& \llbracket e_1, e_2 \rrbracket,\\
l_2 (e \wedge s) =~& \frac{1}{2} \langle e, \mathcal{D} s \rangle,\\
l_2 (s \wedge e) =~& -\frac{1}{2} \langle e, \mathcal{D} s \rangle,\\
l_2 =~& 0, \qquad\text{otherwise}\\
l_3 (e_1 \wedge e_2 \wedge e_3) =~& - T(e_1, e_2, e_3),\\
l_3 =~&  0, \qquad \text{otherwise}
\end{align*}
for $e, e_1, e_2 , e_4 \in \Gamma E$; $s \in \Gamma L$, 
and other maps are zero.
\end{thm}

This $L_\infty$-algebra is similar to \cite{roy-wein} for Courant algebroids. The proof is also similar to the one for Courant algebroids. Hence, we omit the details of the proof. We can restrict the above $L_\infty$-algebra to the truncated complex $\Gamma L \xrightarrow{\mathcal{D}} \Gamma E$ to get a $2$-term $L_\infty$-algebra.

\begin{thm}\label{new-2-term}
Let $(E, [-,-] , \langle -,- \rangle, \rho)$ be an $L$-Courant algebroid. Then the complex
$$\Gamma L \xrightarrow{\mathcal{D}} \Gamma E$$
carries a $2$-term $L_\infty$-algebra with the structure maps
\begin{align*}
l_2 (e_1 \wedge e_2) =~& \llbracket e_1, e_2 \rrbracket,\\
l_2 (e \wedge s) =~& \frac{1}{2} \langle e, \mathcal{D} s \rangle,\\
l_2 (s \wedge e) =~& -\frac{1}{2} \langle e, \mathcal{D} s \rangle,\\
l_3 (e_1 \wedge e_2 \wedge e_3) =~& - T(e_1, e_2, e_3),
\end{align*}
and higher $l_k$'s are zero.
\end{thm}

In the next, we observe few results associated to the $L_\infty$-algebra constructed in Theorem \ref{new-2-term}.

\medskip

\noindent {\bf Observation 1.} It follows that for the $L$-Courant algebroid $\mathbb{D}L = DL \oplus J^1L$ as in Example \ref{omni}, the complex $\Gamma L \xrightarrow{d_{DL}} \Gamma (DL \oplus J^1L)$ inherits a $2$-term $L_\infty$-algebra whose structure maps $l_2, l_3$ are given by the above theorem. In the next, we show that this $2$-term $L_\infty$-algebra can also be seen as a semidirect product $L_\infty$-algebra of a representation up to homotopy.

We will define a representation up to homotopy of the Lie algebra $\Gamma (DL)$ on the $2$-term chain complex
\begin{align*}
(V_1 \xrightarrow{d} V_0) := \Gamma L \xrightarrow{d_{DL}} \Gamma (J^1L).
\end{align*}
Define
\begin{align*}
\mu_0 (\triangle) (\alpha) :=~& \llbracket \triangle, \alpha \rrbracket = \mathcal{L}_\triangle \alpha - \frac{1}{2} d_{DL} (\alpha (\triangle)),\\
\mu_1 (\triangle) (s) :=~& \frac{1}{2} (d_{DL} s)(\triangle), ~~~~ \text{ for } \triangle \in \Gamma (DL), \alpha \in \Gamma (J^1L)   \text{ and } s \in \Gamma L.
\end{align*}
Moreover, we define $\nu : \Gamma (DL) \times \Gamma (DL) \rightarrow Hom (\Gamma (J^1L), \Gamma L)$ by
\begin{align*}
\nu (\triangle, \nabla) (\alpha) :=  T (\triangle, \nabla, \alpha), ~~~ \text{ for all } \triangle, \nabla \in \Gamma (DL),~ \alpha \in \Gamma (J^1L).
\end{align*}

Then we obtain the following \cite{sheng-zhu}.
\begin{prop}
With the above notations
\begin{align*}
(\Gamma L \xrightarrow{d_{DL}} \Gamma (J^1L), \mu_0, \mu_1, \nu)
\end{align*}
defines a representation up to homotopy of the Lie algebra $\Gamma (DL)$.
\end{prop}

Hence, it follows from Proposition \ref{thm-new-2-term} that the complex $\Gamma L \xrightarrow{d_{DL}} \Gamma (DL \oplus J^1L)$ inherits a $2$-term $L_\infty$-algebra with the structure maps
\begin{align*}
l_2 (\triangle + \alpha, \nabla + \beta) =~& [\triangle, \nabla] + \mathcal{L}_\triangle \beta - \mathcal{L}_\nabla \alpha - \frac{1}{2}~ d_{DL} (\beta (\triangle) - \alpha (\nabla)), \\
l_2 (\triangle + \alpha, s) =~& \frac{1}{2} (d_{DL}s) (\triangle), \\
l_2 (s, \triangle + \alpha) =~& - \frac{1}{2} (d_{DL}s) (\triangle), \\
l_3 (\triangle + \alpha, \nabla + \beta, \square + \gamma) =~& - T (\triangle, \nabla, \gamma) + c.p.,
\end{align*}
for $\triangle + \alpha,~ \nabla + \beta,~ \square + \gamma \in \Gamma (DL \oplus J^1L)$ and $s \in \Gamma L.$

Therefore, we have the following.

\begin{prop}\label{semi-direct}
The $2$-term $L_\infty$-algebra induced from the $L$-Courant algebroid $\mathbb{D}L = DL \oplus J^1L$ is same as the semidirect product $L_\infty$-algebra of the above representation up to homotopy of the Lie algebra $\Gamma (DL)$.
\end{prop}

\medskip

\noindent {\bf Observation 2.} Let $L$ be a line bundle over $M$ and consider $(\mathbb{D}L)^p := DL \oplus (\wedge^p (DL)^* \otimes L)$, for $p \geq 0$. For any closed Atiyah $2$-form $B \in \Omega_{cl}^2 (DL,L)$, one can define a Lie algebroid structure on the bundle $(\mathbb{D}L)^0 : = DL \oplus L$ whose bracket and anchor are given by
\begin{align*}
[ (\triangle, s), (\nabla, t)]_B =~& ([\triangle, \nabla], \triangle (t) - \nabla (s) + B (\triangle, \nabla)),\\
\rho (\triangle, s) =~&  \sigma (\triangle),
\end{align*}
for $(\triangle, s), (\nabla, t) \in \Gamma ((\mathbb{D}L)^0)$. One can view the above Lie algebra on the space of sections $\Gamma ((\mathbb{D}L)^0)$ as a $2$-term $L_\infty$-algebra whose underlying complex is given by $(0 \rightarrow \Gamma ((\mathbb{D}L)^0)).$

For $p=1$, we have shown that the bundle $(\mathbb{D}L)^1 := DL \oplus ((DL)^* \otimes L)$ admits an $L$-Courant algebroid structure. Hence, by Theorem \ref{new-2-term}, there is a $2$-term  $L_\infty$-algebra on
$\Gamma L \xrightarrow{\mathcal{D}} \Gamma ((\mathbb{D}L)^1)$. In the following, we show that there is a morphism of $2$-term  $L_\infty$-algebras from  $ (0 \rightarrow \Gamma ((\mathbb{D}L)^0))$ to $(\Gamma L \xrightarrow{\mathcal{D}} \Gamma ((\mathbb{D}L)^1))$.

\begin{thm}\label{morphism-thm}
There is a canonical morphism
$$\phi :   (0 \rightarrow \Gamma ((\mathbb{D}L)^0)) ~\rightsquigarrow ~ (\Gamma L \xrightarrow{\mathcal{D}} \Gamma ((\mathbb{D}L)^1))$$
between $2$-term $L_\infty$-algebras given by
\begin{align*}
\phi_0 :~& \Gamma ((\mathbb{D}L)^0) \rightarrow  \Gamma ((\mathbb{D}L)^1), ~~ (\triangle, s) \mapsto (\triangle, d_{DL} s),\\
\phi_2 :~& \Gamma ((\mathbb{D}L)^0) \times \Gamma ((\mathbb{D}L)^0) \rightarrow \Gamma L, ~~((\triangle, s), (\nabla, t)) \mapsto ~ -\frac{1}{2} (\triangle (t) - \nabla (s)) - B (\triangle, \nabla),
\end{align*}
for $(\triangle, s), (\nabla, t) \in \Gamma ((\mathbb{D}L)^0).$
\end{thm}

\begin{proof}
The proof is straightforward and hence we omit the details. Please see \cite{zambon} for a similar verification.
\end{proof}


\medskip

\noindent {\bf Observation 3.} Let $\omega \in \Omega^3_{cl} (DL,L)$ be a closed Atiyah $3$-form on $L$ and consider the $L$-Courant algebroid $(\mathbb{D}L)_\omega$. For any Atiyah $2$-form $B \in \Omega^2(DL,L)$, one may also consider the $L$-Courant algebroid $(\mathbb{D}L)_{\omega + d_{DL} \omega}$. In the next, we observe that the corresponding $2$-term $L_\infty$-algebras are strictly isomorphic. In other words, the $2$-term $L_\infty$-algebras induced from $L$-Courant algebroids defined by cohomologous Atiyah $3$-forms are isomorphic.

\begin{prop}\label{cohomologous-thm}
There is a strict isomorphism between $2$-term $L_\infty$-algebras induced from the $L$-Courant algebroids $(\mathbb{D}L)_\omega$ and ~$(\mathbb{D}L)_{\omega + d_{DL} B}.$
\end{prop}

\begin{proof}
Let the $2$-term $L_\infty$-algebra induced from the $L$-Courant algebroid $(\mathbb{D}L)_\omega$ is given by $( \Gamma L \xrightarrow{d_{DL}} \Gamma (\mathbb{D}L), l_2, l_3 )$ and induced from the $L$-Courant algebroid $(\mathbb{D}L)_{\omega + d_{DL} B}$ is given by $( \Gamma L \xrightarrow{ d_{DL}} \Gamma (\mathbb{D}L), l'_2, l'_3 )$, where the structure maps are given by Theorem \ref{new-2-term}.

Define $\phi_0 : \Gamma (\mathbb{D}L) \rightarrow \Gamma (\mathbb{D}L)$ and $\phi_1 : \Gamma L \rightarrow \Gamma L$ by
\begin{align*}
\phi_0 (\triangle, \alpha) =~& (\triangle, \alpha + i_\triangle B),\\
\phi_1 (s) =~& s , \text{ for } (\triangle, \alpha ) \in \Gamma (\mathbb{D}L), ~s \in \Gamma L.
\end{align*}
Then $\phi = (\phi_0, \phi_1)$ defines a strict isomorphism from the $2$-term $L_\infty$-algebra induced from the $L$-Courant algebroid $(\mathbb{D}L)_\omega$ to that of $(\mathbb{D}L)_{\omega + d_{DL} B}.$ The verification is straightforward, hence, we omit the details.
\end{proof}

\subsection{Exact $L$-Courant algebroids}

Let $L$ be a line bundle over $M$. In the next, we study exact $L$-Courant algebroids and associate a third Atiyah cohomology class of $L$ to any exact $L$-Courant algebroid. Due to the acyclicity of the Atiyah complex, the corresponding cohomology class is zero.

\begin{defn}
An $L$-Courant algebroid $(E, [-,-] , \langle -,- \rangle, \rho)$ is called exact if
\begin{align}\label{exact-seq}
 0 \xrightarrow{~~~} (DL)^* \otimes L \xrightarrow{~\rho^*} E \xrightarrow{~\rho~} DL \rightarrow{~~~} 0
 \end{align}
is an exact sequence of vector bundles.
\end{defn}

The $L$-Courant algebroids of Examples \ref{omni} and \ref{twisted-omni} are exact.

\begin{defn}
A connection on an exact $L$-Courant algebroid $E$ is a right splitting of (\ref{exact-seq}) which is isotropic. In other words, a connection on $E$ is a bundle map $A : DL \rightarrow E$ satisfying $\rho \circ A = \text{id}_{DL}$ and $\langle A (\triangle), A (\nabla) \rangle = 0$, for all $\triangle, \nabla \in \Gamma(DL)$.
\end{defn}

Such a splitting always exist for an $L$-Courant algebroid. Moreover, if $A$ is a connection and $\theta \in \Omega^2 (DL,L)$ is a Atiyah $2$-form,  the bundle map $(A + \theta) : DL \rightarrow E$ given by
\begin{align}\label{new-connection}
(A + \theta) (\triangle) := A (\triangle) + \rho^* (i_\triangle \theta)
\end{align}
is a new connection. One can show that any two connections on an exact $L$-Courant algebroid must differ by an Atiyah $2$-form as in (\ref{new-connection}).

Given a connection $A$, there is a curvature Atiyah $3$-form $H \in \Omega^3 (DL,L)$ defined by
$$H (\triangle, \nabla, \square) = - \langle \llbracket A (\triangle), A(\nabla) \rrbracket, A (\square) \rangle,~~~ \text{ for } \triangle, \nabla, \square \in \Gamma (DL).$$
Then $H$ is a closed Atiyah $3$-form on $L$ and hence exact. Using the bundle isomorphism $A \oplus \rho^* : DL \oplus ((DL)^* \otimes L) \rightarrow E$, we can transfer the $L$-Courant algebroid structure on $\mathbb{D}L = DL \oplus ((DL)^* \otimes L)$. More precisely, for $(\triangle, \alpha ), (\nabla, \beta) \in \Gamma (\mathbb{D}L)$, the bracket, $L$-valued pairing and anchor are given by
\begin{align*}
[ (\triangle, \alpha), (\nabla, \beta) ] =~& ( [\triangle, \nabla], \mathcal{L}_\triangle \beta - i_\nabla d_{DL} \beta - i_\nabla i_\triangle H),\\
\langle (\triangle, \alpha), (\nabla, \beta) \rangle =~& \alpha (\nabla) + \beta (\triangle) ,\\
\rho (\triangle, \alpha) =~& \triangle.
\end{align*}
This is the $H$-twisted $L$-Courant algebroid structure on $ \mathbb{D}L = DL \oplus ((DL)^* \otimes L)$. Finally, note that, if $\theta \in \Omega^2 (DL,L)$ is a Atiyah $2$-form on $L$, then the curvature $3$-form corresponding to the connection $A + \theta$ is given by $H + d_{DL} \theta$. Hence $[H] = [H + d_{DL} \theta ] = 0.$

This approach is similar to the classification of exact Courant algebroids over $M$ by the third de Rham cohomology $H^3_{\text{deR}}(M)$. The corresponding third de Rham class associated to an exact Courant algebroid is called the \v{S}evera class of it \cite{severa}. However, in the case of exact $L$-Courant algebroid, the corresponding class in $H^3(DL,L)$ is zero due to acyclicity of the Atiyah complex.

\subsection{$L$-Dirac structures}
\begin{defn}
Let $(E, [-,-] , \langle - , - \rangle, \rho)$ be an $L$-Courant algebroid. An $L$-Dirac structure of $E$ is a maximally isotropic subbundle $\xi \subset E$ which is involutive in the sense that $[\Gamma \xi, \Gamma \xi ] \subset \Gamma \xi$ (or equivalently $\llbracket \Gamma \xi, \Gamma \xi \rrbracket \subset \Gamma \xi$).
\end{defn}

An $L$-Dirac structure $\xi$ on the $L$-Courant algebroid $E$ naturally inherits a Lie algebroid structure whose Lie bracket is the restriction of the bracket $[-,-]$ (or $\llbracket -, - \rrbracket$) on $\Gamma \xi$ and the anchor is the restriction of $\sigma \circ \rho : E \rightarrow TM$ to $\xi$. An $L$-Dirac structure on the $L$-Courant algebroid $\mathbb{D}L = DL \oplus J^1L$ of Example \ref{omni} is called a Dirac-Jacobi structure \cite{vitag2}. Dirac-Jacobi structures are generalization of Dirac structures \cite{courant}, Jacobi structures \cite{marle} and Wade's $\mathcal{E}^1 (M)$-Dirac structures \cite{wade}.

\begin{defn}
Let $L$ be a line bundle over $M$. A Jacobi structure on $L$ is a Lie bracket $J = \{-,-\} : \Gamma L \times \Gamma L \rightarrow \Gamma L$ which is a derivation in both entries.
\end{defn}

Note that a biderivation $J = \{-,-\}$ can also be interpreted as an $L$-valued skew-symmetric map (denoted by the same notation) $J : \wedge^2 J^1L \rightarrow L$ defined by
$$J (j^1s_1, j^1s_2) = \{s_1, s_2\}, ~~ \text{ for } s_1, s_2 \in \Gamma L.$$
Therefore, the corresponding induced map $J^\sharp : J^1L \rightarrow DL$ is given by $J^\sharp (\alpha)(s) := J (\alpha, j^1s)$, for $\alpha \in \Gamma(J^1L)$ and $s \in \Gamma L$. It can be checked that a biderivation $J = \{-,-\}$ defines a Jacobi structure on $L$ if and only if $\text{Gr }(J^\sharp) := \{ (J^\sharp \alpha, \alpha ) |~ \alpha \in \Gamma (J^1L)\} \subset \mathbb{D}L$ defines an $L$-Dirac structure on the $L$-Courant algebroid $\mathbb{D}L$. In other words, Gr$(J^\sharp)$ defines a Dirac-Jacobi structure on $L$.

Let $\omega \in \Omega^3_{cl} (DL,L)$ be a closed Atiyah $3$-form on $L$. It  follows from Example \ref{twisted-omni} that  $(\mathbb{D}L)_\omega$ is also an $L$-Courant algebroid. An $\omega$-twisted Jacobi structure is a biderivation $J = \{-,-\} : \Gamma L \times \Gamma L \rightarrow \Gamma L$ such that $\text{Gr }(J^\sharp)$ defines an $L$-Dirac structure on the $L$-Courant algebroid $(\mathbb{D}L)_\omega$. It is easy to see that $J = \{-,-\}$ does not satisfy the Jacobi identity. However, if $\triangle_s = \{s, -\} \in \Gamma (DL)$ is the derivation associated to $s \in \Gamma L$, we have
$$\{ s_1, \{s_2, s_3 \} \} + c.p. = \omega (\triangle_{s_1}, \triangle_{s_2}, \triangle_{s_3}).$$

Since any $L$-Dirac structure on an $L$-Courant algebroid gives rise to a Lie algebroid structure, we have the following.

\begin{prop}\label{twisted-j-lie}
Let $J$ be an $\omega$-twisted Jacobi structure on $L$. Then the $1$-jet bundle $J^1L$ inherits a Lie algebroid structure whose bracket and anchor are given by
\begin{align*}
\llbracket \alpha, \beta \rrbracket_\omega :=~& \mathcal{L}_{J^\sharp \alpha} \beta - i_{J^\sharp \beta} d_{DL} \alpha - i_{J^\sharp \beta} i_{J^\sharp \alpha} \omega,\\
\rho :=~& \sigma \circ J^\sharp,
\end{align*}
for $\alpha, \beta \in \Gamma (J^1L)$.
\end{prop}

When $\omega = 0$, one recovers the notion of Jacobi structures on $L$ and the corresponding Lie algebroid structure on $J^1L$ is the standard one associated to a Jacobi structure \cite{marle}.

In the next, we deform a Dirac-Jacobi structure on $L$ by a closed Atiyah $2$-form. Let $\xi \subset \mathbb{D}L$ be a Dirac-Jacobi structure on $L$ and $B \in \Omega^2_{cl} (DL,L)$ be a closed Atiyah $2$-form. Then it is easy to verify that
$$\tau_B (\xi) = \{ (\triangle, \alpha + i_\triangle B) |~ (\triangle, \alpha) \in \xi \} \subset \mathbb{D}L$$
is also a Dirac-Jacobi structure. This Dirac-Jacobi structure is called the gauge transformation of $\xi$. If the Dirac-Jacobi structure $\xi$ is given by $\text{Gr }(J^\sharp)$ for some Jacobi structure $J$, then  $\tau_B (\xi)$ need not be induced from the graph of another Jacobi structure. However, if the bundle map $(Id + \widetilde{B} \circ J^\sharp) : J^1L \rightarrow J^1L$ is invertibe, (where $\widetilde{B} : DL \rightarrow J^1L$ is the bundle map induced by $B$) then $\tau_B (\xi)$ is the graph of a new Jacobi structure $\tau_B (J).$ The Jacobi structure $\tau_B (J)$ is completely determined by $$(\tau_B (J))^\sharp = J^\sharp (Id + \widetilde{B} \circ J^\sharp)^{-1} : J^1L \rightarrow J^1L.$$
In this case, the new Jacobi structure $\tau_B (J)$ is called the gauge transformation of the Jacobi structure $J$. The gauge transformations of Jacobi structures are related to gauge transformations of Poisson structures via the so called Poissonization process. See \cite{das2} for more details when the line bundle $L$ is trivial.

\medskip

\section{Isotropic involutive subbundles of $(\mathbb{D}L)^p$}

Let $L$ be a line bundle over $M$ and consider the bundle
$$(\mathbb{D}L)^p = DL \oplus (\wedge^p (DL)^* \otimes L).$$ 
Then the bundle $(\mathbb{D}L)^p$ carries a symmetric $(\wedge^{p-1} (DL)^* \otimes L)$-valued pairing given by
$$\langle (\triangle, \alpha), (\nabla, \beta) \rangle = i_\triangle \beta + i_\nabla \alpha.$$
Moreover, the space of sections $\Gamma ((\mathbb{D}L)^p)$ carries a bracket (higher Dorfman-Jacobi bracket)
$$[(\triangle, \alpha), (\nabla, \beta)] = ([\triangle, \nabla], ~\mathcal{L}_\triangle \beta - i_\nabla d_{DL} \alpha),$$
for $(\triangle, \alpha), (\nabla, \beta) \in  \Gamma ((\mathbb{D}L)^p)$. The higher Dorfman-Jacobi bracket satisfies the following identities:
\begin{align}
[e_1,[e_2,e_3]] =~& [[e_1,e_2],e_3] + [e_2,[e_1,e_3]], \label{jac-id}\\
[e_1, f e_2] =~& f [e_1,e_2] + (\sigma \circ pr_1 (e_1))(f) e_2, \label{leib-rule}
\end{align}
for $e_1, e_2, e_3 \in \Gamma ((\mathbb{D}L)^p).$

Note that the corresponding skew-symmetrization bracket $\llbracket -, - \rrbracket$ is given by
\begin{align}\label{dorfman-courant-bracket}
\llbracket e_1, e_2 \rrbracket = [ e_1, e_2 ] - \frac{1}{2} ~d_{DL} \langle e_1, e_2 \rangle
\end{align}
and is called the higher Courant-Jacobi bracket.

\begin{defn}
A subbundle $\xi \subset (\mathbb{D} L)^p$ is called
\begin{itemize}
\item[(i)] {\em isotropic} if $\xi \subset \xi^\perp$, where $\xi^\perp = \{ l \in (\mathbb{D}L)^p |~ \langle e, l \rangle = 0 , ~\forall e \in \Gamma \xi \}$,
\item[(ii)] {\em Lagrangian} if $\xi = \xi^\perp$,
\item[(iii)] {\em involutive} if $[\Gamma \xi, \Gamma \xi] \subset \Gamma \xi$.
\end{itemize}
\end{defn}

A subbundle $\xi \subset (\mathbb{D}L)^p$ is called a Dirac-Jacobi structure of order $p$ (or higher Dirac-Jacobi structure) if $\xi$ is Lagrangian and involutive. Note that any Dirac-Jacobi structure of order $1$ is the usual Dirac-Jacobi structure \cite{vitag2}. However, we are only interested in isotropic involutive subbundles of $(\mathbb{D}L)^p$.

In the following we consider some examples of isotropic involutive subbundles of $(\mathbb{D}L)^p$.

\begin{exam}
Let $\omega \in \Omega_{cl}^{p+1}(DL,L)$ be a closed Atiyah $(p+1)$-form. Then
$$\text{Gr}(\omega) := \{ (\triangle, i_\triangle \omega) |~ \triangle \in DL \} \subset (\mathbb{D}L)^p$$
is an isotropic involutive subbundle of $(\mathbb{D}L)^p$. This is isotropic because
$$\langle (\triangle, i_\triangle \omega), (\nabla, i_\nabla \omega) \rangle = i_\triangle i_\nabla \omega + i_\nabla i_\triangle \omega = 0.$$
Moreover, since $\omega$ is closed, we have $\mathcal{L}_\triangle \omega = d (i_\triangle \omega)$. Therefore,
\begin{align*}
[(\triangle, i_\triangle \omega), (\nabla, i_\nabla \omega)] =~& ( [\triangle, \nabla],~ \mathcal{L}_\triangle i_\nabla \omega - i_\nabla d (i_\triangle \omega) )\\
=~& ([\triangle, \nabla],~ \mathcal{L}_\triangle i_\nabla \omega - i_\nabla \mathcal{L}_\triangle \omega) =~ ([\triangle, \nabla],~ i_{[\triangle, \nabla]} \omega).
\end{align*}
Hence $ \text{Gr}(\omega)$ is involutive. Observe that the bundle $\text{Gr}(\omega)$ is isomorphic to $DL$ via the projection onto the first factor. Any isotropic involutive subbundle $\xi \in (\mathbb{D}L)^p$ with this property is given by a closed Atiyah $(p+1)$-form.

Let $\xi \in (\mathbb{D}L)^p$ be an isotropic involutive subbundle which projects isomorphically onto $DL$. Then $\xi = \{ (\triangle, B(\triangle))|~ \triangle \in DL\}$, for some map $B : DL \rightarrow \wedge^p (DL)^* \otimes L$. Since $\xi$ is isotropic, we have the map
$$ DL \otimes DL \rightarrow \wedge^{p-1}(DL)^* \otimes L, ~ \triangle \otimes \nabla \mapsto i_\triangle (B(\nabla))$$
is skew-symmetric in $\triangle, \nabla$. Therefore, $B(\triangle) = i_\triangle \omega$ for some Atiyah $(p+1)$-form $\omega$. Moreover, $\xi$ is involutive shows that $\omega$ is closed.
\end{exam}

One can generalize the above example in the following way. Let $\omega \in \Omega^{p+1} (DL,L)$ be a Atiyah $(p+1)$-form on $L$ and let $S \subset DL$ be an involutive subbundle of $DL$ such that $(d_{DL} \omega)_{\wedge^3 S \otimes \wedge^{p-1} DL} = 0$. Then 
$$\{ (\triangle, i_\triangle \omega + \alpha )|~ \triangle  \in S, \alpha \in S^0 \}$$
is an isotropic involutive subbundle of $(\mathbb{D}L)^p$.

Isotropic involutive subbundles of $(\mathbb{D}L)^p$ are related to Lie algebroids.
\begin{prop}
Let $\xi \subset (\mathbb{D}L)^p$ be an isotropic involutive subbundle of $(\mathbb{D}L)^p$. Then the triple $(\xi, [-,-], \sigma \circ pr_1)$ forms a Lie algebroid, where $\sigma : DL \rightarrow TM$ is the symbol map.
\end{prop}
\begin{proof}
It follows from (\ref{dorfman-courant-bracket}) that the restriction of the bracket $[-,-]$ on $\Gamma \xi$ is skew-symmetric. Therefore, it also satisfy the Jacobi identity beacause of (\ref{jac-id}). Finally, the Leibniz rule also holds beacause of (\ref{leib-rule}).
\end{proof}

\medskip

\section{Isotropic involutive subbundles and $L$-infinity algebras}

Let $\xi \subset (\mathbb{D}L)$ be a Dirac-Jacobi structure on $L$. A section $s \in \Gamma L$ is called Hamiltonian if there exists a derivation $\triangle_s \in \Gamma (DL)$ such that $(\triangle_s, d_{DL} s) \in \Gamma \xi$. Note that, $\triangle_s$ is unique up to smooth sections of $\xi \cap (DL \oplus \{0\})$. The space of Hamiltonian sections are denoted by $\Gamma_{Ham} (L).$

For any $s_1, s_2 \in \Gamma_{Ham} (L)$, one can define a bracket
$\{ s_1, s_2 \} := i_{\triangle_{s_1}} d_{DL} s_2.$ Then the bracket $\{-,-\}$ defines a Lie algebra structure on $\Gamma_{Ham} (L)$. This Lie algebra is called the algebra of observables of the Dirac-Jacobi structure $\xi$. In the next, we consider isotropic invariant subbundles of $(\mathbb{D}L)^p$ and their corresponding algebra of observables.

Let $L$ be a line bundle over $M$ and $\xi \subset (\mathbb{D}L)^p$ be a isotropic involutive subbundle of $(\mathbb{D}L)^p$.

\begin{defn}
An Atiyah $(p-1)$-form $\alpha \in \Omega^{p-1}(DL,L)$ is called Hamiltonian if there exists a section $\triangle_\alpha \in \Gamma (DL)$ such that $(\triangle_\alpha, d_{DL} \alpha) \in \Gamma \xi$.
\end{defn}

In this case, $\triangle_\alpha$ is called a Hamiltonian derivation associated to $\alpha$. Note that, $\triangle_\alpha$ is unique only up to smooth sections of $\xi \cap (DL \oplus \{0\})$.

The set of Hamiltonian Atiyah $(p-1)$-forms are denoted by $\Omega^{p-1}_{Ham} (DL,L)$. Then there is a bracket $\{-,-\}$ on $\Omega^{p-1}_{Ham}(DL,L)$ defined by
$$\{ \alpha, \beta \} := i_{\triangle_\alpha} d_{DL} \beta,$$
where $\triangle_\alpha$ is any Hamiltonian derivation associated to $\alpha$. One can easily verify that the above bracket does'nt depend on the choice of $\triangle_\alpha$. Moreover, since
\begin{align*}
[(\triangle_\alpha, d_{DL} \alpha) , (\triangle_\beta, d_{DL} \beta) ] =~&  ([\triangle_\alpha, \triangle_\beta],~ \mathcal{L}_{\triangle_\alpha} d_{DL} \beta)\\
=~& ( [\triangle_\alpha, \triangle_\beta],~ d_{DL} \{\alpha, \beta \}  ),
\end{align*}
it follows that $\{\alpha, \beta \} \in \Omega^{p-1}_{Ham} (DL,L)$ with a Hamiltonian derivation $[\triangle_\alpha, \triangle_\beta ]$.

The bracket $\{-,-\}$ is skew-symmetric as
$$\{\alpha, \beta \} + \{ \beta , \alpha \} = \langle (\triangle_\alpha, d_{DL} \alpha), (\triangle_\beta, d_{DL} \beta) \rangle = 0.$$
The bracket $\{-,-\}$ does not satisfy the Jacobi identity (in general), however, it does up to an exact $(p-1)$-Atiyah form.

\begin{lemma}
For $\alpha, \beta, \gamma \in \Omega^{p-1}_{Ham} (DL,L)$, we have
$$\{\alpha, \{ \beta, \gamma \} \} + \{ \beta, \{ \gamma, \alpha \} \} + \{ \gamma, \{\alpha, \beta \} \} = - d_{DL} (i_{\triangle_\alpha} \{ \beta, \gamma \}).$$
\end{lemma}

\begin{proof}
Since $\xi$ is isotropic and involutive, we have
\begin{align*}
0 =~& \langle [(\triangle_\alpha, d_{DL} \alpha), (\triangle_\beta, d_{DL} \beta)], (\triangle_\gamma, d_{DL} \gamma)     \rangle \\
=~& \langle ([\triangle_\alpha, \triangle_\beta], d_{DL} \{\alpha, \beta \}), (\triangle_\gamma, d_{DL} \gamma)  \rangle \\
=~& i_{[\triangle_\alpha, \triangle_\beta]} d_{DL} \gamma + i_{\triangle_\gamma} d_{DL} \{ \alpha, \beta \} \\
=~& \mathcal{L}_{\triangle_\alpha} i_{\triangle_\beta} d_{DL} \gamma - i_{\triangle_\beta} \mathcal{L}_{\triangle_\alpha} d_{DL} \gamma + \{ \gamma, \{ \alpha, \beta \} \}\\
=~& \mathcal{L}_{\triangle_\alpha} \{ \beta, \gamma \} - i_{\triangle_\beta} d_{DL} i_{\triangle_\alpha} d_{DL} \gamma + \{ \gamma, \{ \alpha, \beta \} \}\\
=~& d_{DL} i_{\triangle_\alpha} \{ \beta, \gamma \}   + i_{\triangle_\alpha} d_{DL} \{\beta, \gamma \} - i_{\triangle_\beta} d_{DL} \{ \alpha, \gamma \} + \{ \gamma, \{ \alpha, \beta \} \}\\
=~& d_{DL} i_{\triangle_\alpha} \{ \beta, \gamma \}  + \{\alpha, \{ \beta, \gamma \} \} + \{ \beta, \{ \gamma, \alpha \} \} + \{ \gamma, \{\alpha, \beta \} \}.
\end{align*}
\end{proof}

\begin{remark}
A contact manifold is a manifold $M$ equipped with a contact distribution $\mathcal{H} \subset TM$. By a contact distribution $\mathcal{H} \subset TM$,
we mean a maximally non-integrable hyperplane distribution $\mathcal{H}$ on $M$. Note that any hyperplane distribution $\mathcal{H}$ can be viewed as the kernel of the vector valued $1$-form $\theta : TM \rightarrow L$, where $L = TM / \mathcal{H}$. The hyperplane distribution $\mathcal{H}$ defines a contact distribution if and only if $\omega := d_{DL} (\sigma^* \theta) \in \Omega^2(DL,L)$ is a closed non-degenerate Atiyah $2$-form on $L$ \cite{vitag2}. In this case the corresponding isotropic involutive subbundle is given by
$$\text{Gr}(\omega) = \{ (\triangle, i_\triangle \omega) |~ \triangle \in \Gamma (DL) \} \subset DL \otimes J^1L = (\mathbb{D}L)^1.$$
Since $\omega$ is non-degenerate, we have $\Omega^0_{\text{Ham}} (DL,L) = \Gamma L$. Moreover, the corresponding bracket $\{-,-\}$ on  $\Omega^0_{\text{Ham}} (DL,L) = \Gamma L$ satisfies the Jacobi identity. In fact, this bracket defines a Jacobi structure on $L$. 
\end{remark}

An alternative proof of the above proposition is given by the next lemma. This lemma will be used to construct an $L_\infty$-algebra associated to any isotropic involutive subbundle $\xi \subset (\mathbb{D}L)^p.$

\begin{lemma}\label{useful-lemma}
Let $\xi \subset (\mathbb{D}L)^p$ be an isotropic involutive subbundle of $(\mathbb{D}L)^p$. Then for any $n \geq 3$ and $\alpha_1, \ldots, \alpha_n \in \Omega^{p-1}_{Ham} (DL,L)$, we have
\begin{align*}
 (-1)^{n+1}~ d_{DL} & (i_{\triangle_{\alpha_n}} \cdots i_{\triangle_{\alpha_3}} \{\alpha_1, \alpha_2 \}) \\
& = \sum_{2 \leq i < j \leq n} (-1)^{i+j-1} ~i_{\triangle_{\alpha_n}} \cdots \widehat{i_{\triangle_{\alpha_j}}} \cdots \widehat{i_{\triangle_{\alpha_i}}} \cdots i_{\triangle_{\alpha_2}} \{\{a_i, a_j\}, a_1 \} \\
& + \sum_{3 \leq j \leq n} (-1)^j~ i_{\triangle_{\alpha_n}} \cdots \widehat{i_{\triangle_{\alpha_j}}} \cdots i_{\triangle_{\alpha_3}} \{\{a_1, a_j\}, a_2 \} + i_{\triangle_{\alpha_n}} \cdots i_{\triangle_{\alpha_4}} \{\{\alpha_1, \alpha_2\}, \alpha_3\}.
\end{align*}
\end{lemma}

In \cite{rogers} Rogers associate an $n$-term $L_\infty$-algebra to any $n$-plectic manifold. Later on, Zambon construct an $L_\infty$-algebra to any higher Dirac structures \cite{zambon}. Inspired from their result, we also obtain an $L_\infty$-algebra to any isotropic involutive subbundle of $(\mathbb{D}L)^p$.

\begin{thm}\label{l-inf-iso-inv}
Let $\xi \subset (\mathbb{D}L)^p$ be an isotropic involutive subbundle of $(\mathbb{D}L)^p.$ Then there is a $p$-term $L_\infty$-algebra $(\mathcal{A}, l_k)$ whose underlying chain complex $(\mathcal{A}_{p-1} \xrightarrow{l_1} \mathcal{A}_{p-2} \xrightarrow{l_1} \cdots \xrightarrow{l_1} \mathcal{A}_1 \xrightarrow{l_1} \mathcal{A}_0)$ is given by
$$\Gamma L \xrightarrow{d_{DL}} \Omega^1(DL,L) \xrightarrow{d_{DL}} \Omega^2(DL,L) \xrightarrow{d_{DL}} \cdots \cdot \xrightarrow{d_{DL}} \Omega^{p-2} (DL,L) \xrightarrow{d_{DL}} \Omega^{p-1}_{Ham}(DL,L).$$
The higher maps $l_k : \mathcal{A}^{\otimes k} \rightarrow \mathcal{A}$, for $ 2 \leq k \leq p+1$, are non-zero only on $\mathcal{A}_0^{\otimes k}$ and they are given by
$$
l_k (\alpha_1, \ldots, \alpha_k) = \kappa(k)~ i_{\triangle_{\alpha_k}} \cdots i_{\triangle_{\alpha_3}} \{\alpha_1, \alpha_2 \},$$
for $\alpha_1, \ldots, \alpha_k \in \mathcal{A}_0 = \Omega^{p-1}_{Ham}(DL,L),$ where $\kappa (k) = (-1)^{\frac{k}{2} + 1}$ if $k$ is even and $\kappa (k) = (-1)^{\frac{k-1}{2}}$ if $k$ is odd.
\end{thm}

\begin{proof}
It is easy to see that the higher maps $l_k's$ are skew-symmetric. Moreover, since the bracket $\{-,-\}$ is independent of the choice of Hamiltonian derivation, it follows that the higher maps $l_k$'s are well defined. One can easily see that the map $l_k$ has degree $k-2$, for $k \geq 1$. Thus, we only need to verify that $l_k$'s satisfy higher Jacobi identities.

The higher Jacobi identity for $n=1$ is equivalent to $l_1^2 = 0$. This is true in this case as $l_1 = d_{DL}$ is the Lie algebroid differential. To prove higher Jacobi identities for $n \geq 2$ in our case, we observe that the $k$-multilinear map $l_k$ vanishes when one of its entries have positive degree. Therefore, if $i+j= n+1$ and $j \in \{ 2, \ldots, n-2 \}$, the term
$$l_j (l_i (a_{\sigma(1)}, \ldots, a_{\sigma (i)}), a_{\sigma (i+1)}, \ldots, a_{\sigma (n)}) = 0,$$
as $l_i (a_{\sigma(1)}, \ldots, a_{\sigma (i)})$ lies in positive degree. On the other hand, if $j = n$, we have
$$l_n ( l_1 (a_{\sigma (1)}), a_{\sigma (2)}, \ldots, a_{\sigma (n)}) = 0.$$
This follows from the fact that if $a_{\sigma (1)} \in \mathcal{A}_0$, then  $l_1 (a_{\sigma (1)}) = 0$ and if $a_{\sigma (1)} \in \mathcal{A}_1$, then $l_1 (a_{\sigma (1)}) = d_{DL} (a_{\sigma (1)})$ has its Hamiltonian derivation vanish. Therefore, the only non-zero terms in the higher Jacobi identity occurs for $j = 1$ and $j = n-1$.

\medskip

{\em Case 1. ($n = 2$)} In this case, the summation in Equation (\ref{hl1-iden}) reduces to $l_1 (l_2 (a_1, a_2))$, which is zero by degree reasons.

\medskip

{\em Case 2. ($n \geq 3$)} It is enough to assume that all $\alpha_i$'s are in $\mathcal{A}_0 = \Omega^{p-1}_{Ham} (DL,L)$. The summation  in Equation (\ref{hl1-iden}) reduces to 
$$l_1 (l_n (\alpha_1, \ldots, \alpha_n)) + \sum_{\sigma \in Sh(2, n-2)} \text{sgn}(\sigma)~ \epsilon (\sigma)~ l_{n-1} (\{a_{\sigma (1)}, a_{\sigma (2)} \}, a_{\sigma (3)}, \ldots, a_{\sigma (n)}).$$
Using the explicit unshuffles and definitions of higher maps, we get the above summation as
\begin{align*}
& l_1 (l_n (\alpha_1, \ldots, \alpha_n)) + \sum_{1 \leq i < j \leq n} (-1)^{i+j-1} l_{n-1} (\{a_i, a_j\}, a_1, \ldots, \widehat{a_i}, \ldots, \widehat{a_j}, \ldots, a_n) \\
&=~ l_1 (l_n (\alpha_1, \ldots, \alpha_n)) + l_{n-1} (\{\alpha_1, \alpha_2 \}, \alpha_3, \ldots, \alpha_n) \\
&+ \sum_{3 \leq j \leq n} (-1)^j ~l_{n-1} (\{a_1, a_j\}, a_2, \ldots, \widehat{a_j}, \ldots, a_n) \\
&+ \sum_{2 \leq i < j \leq n} (-1)^{i+j-1} ~l_{n-1} (\{a_i, a_j\}, a_1, \ldots, \widehat{a_i}, \ldots, \widehat{a_j}, \ldots, a_n) \\
&= \kappa(n)~ d_{DL} (i_{\triangle_{\alpha_n}} \cdots i_{\triangle_{\alpha_3}} \{\alpha_1, \alpha_2 \})
+ \kappa(n-1) ~\bigg[ i_{\triangle_{\alpha_n}} \cdots i_{\triangle_{\alpha_4}} \{\{\alpha_1, \alpha_2\}, \alpha_3\} \\
& + \sum_{3 \leq j \leq n} (-1)^j~ i_{\triangle_{\alpha_n}} \cdots \widehat{i_{\triangle_{\alpha_j}}} \cdots i_{\triangle_{\alpha_3}} \{\{a_1, a_j\}, a_2 \} \\
& + \sum_{2 \leq i < j \leq n} (-1)^{i+j-1} i_{\triangle_{\alpha_n}} \cdots \widehat{i_{\triangle_{\alpha_j}}} \cdots \widehat{i_{\triangle_{\alpha_i}}} \cdots i_{\triangle_{\alpha_2}} \{\{a_i, a_j\}, a_1 \} \bigg].
 \end{align*}
This term vanishes from Lemma \ref{useful-lemma}. Hence the higher Jacobi identities hold.
\end{proof}

The above $p$-term $L_\infty$-algebra is called the algebra of observables associated to the isotropic involutive subbundle $\xi$.

\begin{remark}
Let $B \in \Omega^2_{\text{cl}} (DL,L)$ be a closed Atiyah $2$-form. Then
$$\Phi_B : DL \oplus J^1L \rightarrow DL \oplus J^1L, ~ (\triangle, \alpha) \mapsto (\triangle, \alpha + i_\triangle B)$$
defines an automorphism of the $L$-Courant algebroid $\mathbb{D}L = DL \oplus J^1L$. This is called the gauge transformation by $B$. Therefore, it acts on the set of all Dirac-Jacobi structures on $L$. However, the Lie algebra of observables of the corresponding Dirac-Jacobi structures are not isomorphic (unless $B = 0$).
A similar situation also holds for higher dimensions. Namely, if ${\bf B} \in \Omega^{p+1}_{\text{cl}} (DL,L)$ is a closed Atiyah $(p+1)$-form on $L$, then the gauge transformation
$$\Phi_{\bf B} : (\mathbb{D}L)^p  \rightarrow (\mathbb{D}L)^p, ~ (\triangle, \alpha) \mapsto (\triangle, \alpha + i_\triangle {\bf B})$$
acts on the set of isotropic involutive subbundles of $(\mathbb{D}L)^p$. However, it does'nt induce an isomorphism on the corresponding algebra of observables.
\end{remark}

In the next, we define a transformation of isotropic involutive subbundles of $(\mathbb{D}L)^p$ which induces an isomorphism on the corresponding algebra of observables.
\begin{lemma}
For any $\lambda \in \mathbb{R} \setminus \{0\}$, we define a map $m_\lambda : (\mathbb{D}L)^p \rightarrow (\mathbb{D}L)^p$ by
$$m_\lambda (\triangle, \alpha) = (\triangle, \lambda \alpha).$$
If $\xi \subset (\mathbb{D}L)^p$ is an isotropic involutive subbundle of $(\mathbb{D}L)^p$, then $m_\lambda (\xi)$ is also an isotropic involutive subbundle of $(\mathbb{D}L)^p$. Moreover, the corresponding $p$-term $L_\infty$-algebras as in Theorem \ref{l-inf-iso-inv} are strictly isomorphic.
\end{lemma}

\begin{proof}
For $(\triangle, \alpha), (\nabla, \beta) \in \Gamma \xi$, we have
$$\langle m_\lambda (\triangle, \alpha), m_\lambda (\nabla, \beta) \rangle = \lambda \langle (\triangle, \alpha),(\nabla, \beta) \rangle = 0.$$
Moreover, one can easily verify that
$$ [m_\lambda (\triangle, \alpha), m_\lambda (\nabla, \beta)] = m_\lambda [(\triangle, \alpha), (\nabla, \beta)].$$
This shows that $m_\lambda (\xi)$ is an isotropic involutive subbundle of $(\mathbb{D}L)^p$.

To prove the last part, let $\Omega^{p-1}_{Ham, \xi} (DL,L)$ and $\Omega^{p-1}_{Ham, m_\lambda(\xi)} (DL,L)$ denote the space of corresponding Hamiltonian Atiyah $(p-1)$-forms. Observe that, if $\alpha \in \Omega^{p-1}_{Ham, \xi} (DL,L)$ with a Hamiltonian derivation $\triangle_\alpha$, then $\lambda \alpha \in \Omega^{p-1}_{Ham, m_\lambda(\xi)} (DL,L)$ with the same $\triangle_\alpha$ as a Hamiltonian derivation. Moreover, for $\alpha, \beta \in  \Omega^{p-1}_{Ham, \xi} (DL,L)$,
$$\{ \lambda \alpha, \lambda \beta \}_{m_\lambda (\xi)} = i_{\triangle_\alpha} d_{DL} (\lambda \beta) = \lambda \{ \alpha, \beta \}_\xi.$$
Define a degree zero map $f$ between the corresponding $L_\infty$-algebras by the multiplication of $\lambda$. Then it is easy to see that $f$ defines a strict isomorphism between them.
\end{proof}

\subsection{Relation between two $L_\infty$-algebras}

Let $\omega \in \Omega^3_{\text{cl}} (DL,L)$ be a closed Atiyah $3$-form on $L$. Then
$$\text{Gr}(\omega) = \{ (\triangle, i_\triangle \omega)|~ \triangle \in DL \} \subset (\mathbb{D}L)^2$$
is an isotropic involutive subbundle of $(\mathbb{D}L)^2$. Hence it follows from Theorem \ref{l-inf-iso-inv} that the complex $\Gamma L \xrightarrow{d_{DL}} \Omega^1_{\text{Ham}} (DL,L)$ carries a $2$-term $L_\infty$-algebra structure, where $\Omega^1_{\text{Ham}} (DL,L)$ is the set of all $\alpha \in \Omega^1 (DL,L)$ for which there exists a derivation $\triangle_\alpha \in \Gamma (DL)$ with $(\triangle_\alpha, d_{DL} \alpha) \in \Gamma (\text{Gr}(\omega)).$

On the other hand, given a closed Atiyah $3$-form $\omega \in \Omega^3_{\text{cl}} (DL,L)$, we have the $L$-Courant algebroid $(\mathbb{D}L)_\omega$ as in Example \ref{twisted-omni}. Hence, by Theorem \ref{new-2-term}, there is a $2$-term $L_\infty$-algebra on the complex $\Gamma L \xrightarrow{\mathcal{D} = d_{DL}} \Gamma (\mathbb{D}L)$. In the following, we show that when $\omega$ is non-degenerate, the $2$-term $L_\infty$-algebra $\Gamma L \xrightarrow{d_{DL}} \Omega^1_{\text{Ham}} (DL,L)$ naturally embedds into $\Gamma L \xrightarrow{\mathcal{D} = d_{DL}} \Gamma (\mathbb{D}L)$.

We first need the following observations.

\begin{itemize}
\item For $\alpha, \beta \in \Omega^1_{Ham} (DL,L)$ with Hamiltonian derivations $\triangle_\alpha$ and $\triangle_\beta$, respectively, we have
\begin{align}
\mathcal{L}_{\triangle_\alpha} \beta =~& \{ \alpha, \beta \} + d_{DL} i_{\triangle_\alpha} \beta, \label{l-bracket}\\
\mathcal{L}_{\triangle_\alpha} \beta - \mathcal{L}_{\triangle_\beta} \alpha =~& 2 (\{ \alpha, \beta \} + d_{DL} B(\alpha, \beta)), \label{l-minus-l}
\end{align}
where $B (\alpha, \beta) = \frac{1}{2} (i_{\triangle_\alpha} \beta - i_{\triangle_\beta} \alpha).$ The first identity follows from the definition of the Lie derivative  and the second identity follows from the first one.
\item When $\omega$ is non-degenerate, the bracket of Hamiltonian Atiyah $1$-forms are given by
\begin{align}\label{w-bracket}
\{ \alpha, \beta \} = i_{\triangle_\alpha} d_{DL} \beta = i_{\triangle_\alpha} i_{\triangle_\beta} \omega,
\end{align}
as $(\triangle_\beta, d_{DL} \beta) \in \text{Gr}(\omega)$, where $\triangle_\alpha$ and $\triangle_\beta$ are unique Hamiltonian derivations associated to $\alpha$ and $\beta$, respectively.
\end{itemize}

Moreover, we need the following lemma which will ease some of the next computations.

\begin{lemma}\label{some-lemma-cp}
For any $\alpha, \beta, \gamma \in \Omega^1_{Ham} (DL,L)$ with Hamiltonian derivations $\triangle_\alpha, \triangle_\beta$ and $\triangle_\gamma$, respectively, we have
\begin{align*}
i_{[\triangle_\alpha, \triangle_\beta]} \gamma + c.p. = 3~ i_{\triangle_\alpha} i_{\triangle_\beta} i_{\triangle_\gamma} \omega + 2 \big[   i_{\triangle_\alpha} d_{DL} B (\beta, \gamma) + c.p. \big].
\end{align*}
\end{lemma}

\begin{proof}
We have
\begin{align*}
i_{[\triangle_\alpha, \triangle_\beta]} \gamma =~& \mathcal{L}_{\triangle_\alpha} i_{\triangle_\beta} \gamma - i_{\triangle_\beta} \mathcal{L}_{\triangle_\alpha} \gamma \\
=~& i_{\triangle_\alpha} d_{DL} i_{\triangle_\beta} \gamma - i_{\triangle_\beta} (\{\alpha, \gamma \} + d_{DL} i_{\triangle_\alpha} \gamma )  \qquad \text{(by (\ref{l-bracket}))}\\
=~& i_{\triangle_\alpha} d_{DL} i_{\triangle_\beta} \gamma + i_{\triangle_\beta} i_{\triangle_\gamma} i_{\triangle_\alpha} \omega - i_{\triangle_\beta} d_{DL} i_{\triangle_\alpha} \gamma.   \qquad \text{(by (\ref{w-bracket}))}
\end{align*}
Hence,
\begin{align*}
i_{[\triangle_\alpha, \triangle_\beta]} \gamma + c.p. =~& 3~ i_{\triangle_\alpha} i_{\triangle_\beta} i_{\triangle_\gamma} \omega +~ \big[ i_{\triangle_\alpha} d_{DL} (i_{\triangle_\beta} \gamma - i_{\triangle_\gamma} \beta) + c.p.   \big] \\
=~& 3~ i_{\triangle_\alpha} i_{\triangle_\beta} i_{\triangle_\gamma} \omega +~ 2 \big[ i_{\triangle_\alpha} d_{DL} B (\beta, \gamma) + c.p. \big].
\end{align*}
\end{proof}

\begin{thm}\label{inj-morphism}
Let $\omega \in \Omega_{cl}^3 (DL, L)$ be a closed non-degenerate Atiyah $3$-form on $L$. Then there is an injective morphism
$$\phi : (\Gamma L \xrightarrow{d_{DL}} \Omega^1_{\text{Ham}} (DL,L))~ \rightsquigarrow ~ (\Gamma L \xrightarrow{d_{DL}} \Gamma( \mathbb{D}L))$$
between $2$-term $L_\infty$-algebras.
\end{thm}

\begin{proof}
Let $\mathcal{A} = (\Gamma L \xrightarrow{d_{DL}} \Omega^1_{\text{Ham}} (DL,L))$ and $\mathcal{A}' = (\Gamma L \xrightarrow{d_{DL}} \Gamma( \mathbb{D}L))$ be the $2$-term $L_\infty$-algebras. Since $\omega \in \Omega^3_{\text{cl}} (DL,L)$ is non-degenerate, the map
$$\omega_\flat : DL \rightarrow \wedge^2 (DL)^* \otimes L, ~ \triangle \mapsto i_\triangle \omega$$
is injective. Define $\phi_0 : \mathcal{A}_0 \rightarrow \mathcal{A}_0'$ by $\phi_0 (\alpha) = (\triangle_\alpha, - \alpha)$, for $\alpha \in \mathcal{A}_0 = \Omega^1_{\text{Ham}} (DL,L),$ where $\triangle_\alpha \in \Gamma (DL)$ is the Hamiltonian derivation corresponding to $\alpha$. The map $\phi_1 : \mathcal{A}_1 \rightarrow \mathcal{A}_1'$ is given by $\phi_1 (s)= - s$, for $s \in \mathcal{A}_1 = \Gamma L.$ Clearly, the map $\phi_0$ and $\phi_1$ are injective. Finally, the map $\phi_2 : \mathcal{A}_0 \times \mathcal{A}_0 \rightarrow \mathcal{A}_1'$ is given by
$$ \phi_2 (\alpha, \beta) = - \frac{1}{2} (\beta (\triangle_\alpha) - \alpha (\triangle_\beta)),~~ \text{ for } \alpha, \beta \in \mathcal{A}_0 = \Omega^1_{\text{Ham}} (DL,L).$$
It is easy to see that $\phi = (\phi_0, \phi_1) : \mathcal{A} \rightarrow \mathcal{A}'$ defines a chain map and the map $\phi_2$ is skew-symmetric. For $\alpha, \beta \in \mathcal{A}_0 = \Omega^1_{\text{Ham}} (DL,L)$,
\begin{align*}
l_2' (\phi_0 (\alpha), \phi_0 (\beta)) =~& \llbracket (\triangle_\alpha, - \alpha), (\triangle_\beta, - \beta) \rrbracket_\omega \\
=~& ( [\triangle_\alpha, \triangle_\beta],~ - \mathcal{L}_{\triangle_\alpha} \beta + \mathcal{L}_{\triangle_\beta} \alpha + \frac{1}{2} d_{DL} (\beta (\triangle_\alpha) - \alpha (\triangle_\beta)) - i_{\triangle_\beta} i_{\triangle_\alpha} \omega ) \\
=~& ( [\triangle_\alpha, \triangle_\beta],~ - 2 \{\alpha, \beta \} - 2 d_{DL} B (\alpha, \beta) + d_{DL} B (\alpha, \beta) + \{ \alpha, \beta \}) \qquad  \text{(by ~~~(\ref{l-minus-l}))}\\
=~& ([\triangle_\alpha, \triangle_\beta],~ - \{\alpha, \beta \} + d_{DL} \phi_2 (\alpha, \beta)).
\end{align*}
Therefore,
\begin{align*}
l_2' (\phi_0 (\alpha), \phi_0 (\beta)) - \phi_0 (l_2 (\alpha, \beta)) =~& ([\triangle_\alpha, \triangle_\beta],~ - \{\alpha, \beta \} + d_{DL} \phi_2 (\alpha, \beta)) - ([\triangle_\alpha, \triangle_\beta], - \{\alpha, \beta \}) \\
=~& (0, d_{DL} \phi_2 (\alpha, \beta)) = l_1' (\phi_2 (\alpha, \beta)).
\end{align*}
Similarly, for $\alpha \in \mathcal{A}_0 = \Omega^1_{\text{Ham}} (DL,L)$ and $s \in \mathcal{A}_1 = \Gamma L$,
\begin{align*}
l_2' (\phi_0 (\alpha) , \phi_1 (s)) - \phi_1 l_2 (\alpha, s) =~& l_2' ((\triangle_\alpha, - \alpha), -s) \\
=~& \frac{1}{2} \langle (\triangle_\alpha, - \alpha), (0, -d_{DL} s ) \rangle \\
=~& - \frac{1}{2} (d_{DL}s) (\triangle_\alpha) = \phi_2 (\alpha, d_{DL} s) = \phi_2 (\alpha, l_1 s).
\end{align*}
Thus, it remains to show the last condition of Definition \ref{l-inf-map}. For any $\alpha, \beta, \gamma \in \mathcal{A}_0 = \Omega^1_{\text{Ham}} (DL,L),$
\begin{align*}
l_3' (\phi_0 (\alpha), \phi_0 (\beta) , \phi_0 (\gamma)) =~& - T ((\triangle_\alpha, - \alpha), (\triangle_\beta, - \beta), (\triangle_\gamma, - \gamma)) \\
=~& - \frac{1}{6} \langle \llbracket (\triangle_\alpha, - \alpha),   (\triangle_\beta, - \beta) \rrbracket_\omega, (\triangle_\gamma, - \gamma)  \rangle + c.p. \\
=~& - \frac{1}{6} \langle [\triangle_\alpha, \triangle_\beta],~ - \{\alpha, \beta \} - d_{DL} B(\alpha, \beta) ), (\triangle_\gamma, - \gamma)  \rangle + c.p.\\
=~& \frac{1}{6} \big(  i_{[\triangle_\alpha, \triangle_\beta]} \gamma + i_{\triangle_\gamma} i_{\triangle_\alpha} i_{\triangle_\beta} \omega + i_{\triangle_\gamma} d_{DL} B(\alpha, \beta) \big) + c.p. \\
=~& i_{\triangle_\gamma} i_{\triangle_\alpha} i_{\triangle_\beta} \omega + \frac{1}{2} \big(  i_{\triangle_\gamma} d_{DL} B(\alpha, \beta)  + c.p.  \big)\\
 =~&  - l_3 (\alpha, \beta, \gamma) + \frac{1}{2} \big(   i_{\triangle_\gamma} d_{DL} B(\alpha, \beta)  + c.p.   \big) \\
=~& \phi_1 ( l_3 (\alpha, \beta, \gamma)) + \frac{1}{2} \big(   i_{\triangle_\gamma} d_{DL} B(\alpha, \beta)   + c.p.   \big).
\end{align*}
Moreover,
\begin{align*}
S := \phi_2 (\alpha, l_2 (\beta, \gamma)) + c.p. =~& \phi_2 (\alpha, \{ \beta, \gamma \}) + c.p.\\
=~& - \frac{1}{2} (i_{\triangle_\alpha} \{ \beta, \gamma\} - i_{[\triangle_\beta, \triangle_\gamma]} \alpha) + c.p. \\
=~& - \frac{1}{2} (i_{\triangle_\alpha} i_{\triangle_\beta} i_{\triangle_\gamma} \omega - i_{[\triangle_\beta, \triangle_\gamma]} \alpha) + c.p. \\
=~&  - \frac{1}{2} \big[ 3 i_{\triangle_\alpha} i_{\triangle_\beta} i_{\triangle_\gamma} \omega - 3 i_{\triangle_\alpha} i_{\triangle_\beta} i_{\triangle_\gamma} \omega - 2 [ i_{\triangle_\alpha} d_{DL} B (\beta, \gamma) + c.p. ]    \big] \\
& \qquad \qquad \qquad \qquad \qquad \qquad \text{(by Lemma \ref{some-lemma-cp})} \\
=~& i_{\triangle_\alpha} d_{DL} B (\beta, \gamma) + c.p. 
\end{align*}
and
\begin{align*}
T := l_2' (\phi_0 (\alpha), \phi_2 (\beta, \gamma)) + c.p. =~& l_2' (  (\triangle_\alpha, - \alpha), - B (\beta, \gamma)) + c.p. \\
=~& \frac{1}{2} \langle (\triangle_\alpha, - \alpha), (0, - d_{DL} B (\beta, \gamma)) + c.p. \\
=~& - \frac{1}{2} \big[ i_{\triangle_\alpha} d_{DL} B (\beta, \gamma) + c.p. \big]. 
\end{align*}
Therefore,
\begin{align*}
l_3' (\phi_0 (\alpha), \phi_0 (\beta) , \phi_0 (\gamma)) - \phi_1 (l_3 (\alpha, \beta, \gamma))
= \frac{1}{2} \big[ i_{\triangle_\alpha} d_{DL} B (\beta, \gamma) + c.p. \big] = S + T.
\end{align*}
Hence the proof.
\end{proof}

\subsection{dg Leibniz algebras}

In the previous subsection, we associate an $L_\infty$-algebra to any isotropic involutive subbundle of $(\mathbb{D}L)^p$. For a closed Atiyah $(p+1)$-form $\omega \in \Omega^{p+1}_{cl} (DL,L)$ on $L$, the graph $\text{Gr}(\omega) \subset (\mathbb{D}L)^p$
defines a isotropic involutive subbundle. Hence, by Theorem \ref{l-inf-iso-inv} there is a $p$-term $L_\infty$-algebra.

In this section, we show that if $\omega$ is non-degenerate, then there is also a dg (differential graded) Leibniz algebra structure on the same chain complex.
\begin{defn}
A dg Leibniz algebra is a graded vector space $\mathcal{A} = \oplus \mathcal{A}_i$ together with a degree $-1$  differential $\delta$ and a degree $0$ bracket $[-,-] : \mathcal{A} \otimes \mathcal{A} \rightarrow \mathcal{A}$ satisfying
\begin{itemize}
\item (derivation rule) $\delta [a,b] = [\delta a , b] + (-1)^{|a|} [a, \delta b],$
\item (graded Leibniz identity) $[a,[b,c]] = [[a,b],c] + (-1)^{|a||b|} [ b, [a,c]],$
\end{itemize}
for $a,b, c \in \mathcal{A}.$
\end{defn}

\begin{thm}\label{dg-leibniz-thm}
Let $\omega \in \Omega^{p+1}_{cl} (DL,L)$ be a closed non-degenerate Atiyah $(p+1)$-form on $L$. Then there is a dg Leibniz algebra whose underlying chain complex is given by
$$\Gamma L \xrightarrow{d_{DL}} \Omega^1(DL,L) \xrightarrow{d_{DL}} \cdots \cdot \xrightarrow{d_{DL}} \Omega^{p-2}(DL,L) \xrightarrow{d_{DL}} \Omega^{p-1}_{\text{Ham}} (DL,L)$$
and the bracket is given by
\[
[\alpha, \beta ]  = \begin{cases} \mathcal{L}_{\triangle_\alpha} \beta   ~~~~~~  & \mbox{if } \alpha \in \mathcal{A}_0 =  \Omega^{p-1}_{\text{Ham}} (DL,L),\\
0 ~~~~~ & \text{otherwise}
\end{cases}
\]
where $\triangle_\alpha$ is a Hamiltonian derivation associated to $\alpha$.
\end{thm}

\begin{proof}
Let $\alpha, \beta \in A_0 = \Omega^{p-1}_{\text{Ham}} (DL,L)$. Then we have
$$d_{DL} [\alpha, \beta] = d_{DL} \mathcal{L}_{\triangle_\alpha} \beta = d_{DL} i_{\triangle_\alpha} d_{DL} \beta = d_{DL} \{\alpha, \beta \}.$$
This shows that $[\alpha, \beta] \in A_0 = \Omega^{p-1}_{\text{Ham}} (DL,L)$ with a Hamiltonian derivation given by $[\triangle_\alpha, \triangle_\beta].$ It is also clear that the bracket $[-,-]$ has degree $0$.

Next, we prove that the above defined bracket satisfies
$$d_{DL}[\alpha, \beta] = [d_{DL} \alpha, \beta ] + (-1)^{|\alpha|} [\alpha, d_{DL} \beta],~~ \text{ for all } \alpha, \beta \in A.$$
If $|\alpha| > 1$, then both sides of the above identity vanishes.  If $|\alpha| =1$, then the above identity reduces $[d_{DL} \alpha, \beta]=0$ which holds as the Hamiltonian derivation corresponding to $d_{DL}\alpha$ is zero. Finally, if $|\alpha| = 0$, then the above identity is equivalent to $d_{DL} \mathcal{L}_{\triangle_\alpha} \beta = \mathcal{L}_{\triangle_\alpha} d_{DL} \beta$ which holds automatically in a Lie algebroid.

Next, we prove the graded Leibniz identity of the bracket. For $|\alpha| > 0$ or $|\beta| > 0$, it follows from the definition of the bracket that both sides of the identity
$$[\alpha, [\beta, \gamma ]] = [[\alpha, \beta], \gamma ] + (-1)^{|\alpha| |\beta|} [\beta, [\alpha, \gamma]]$$
vanishes. If $|\alpha| = |\beta| = 0$, then
\begin{align*}
[\alpha, [\beta, \gamma]] = \mathcal{L}_{\triangle_\alpha} \mathcal{L}_{\triangle_\beta} \gamma =~& \mathcal{L}_{[\triangle_\alpha, \triangle_\beta]} \gamma + \mathcal{L}_{\triangle_\beta} \mathcal{L}_{\triangle_\alpha} \gamma \\
=~& [[\alpha, \beta], \gamma ] + [\beta, [\alpha, \gamma]].
\end{align*}
Hence the proof.
\end{proof}

\noindent {\bf Concluding remarks.} The notion of $L$-Courant algebroids are generalization of Courant algebroids in the realm of contact geometry. More precisely, like Courant algebroids are symplectic NQ-manifolds of degree $2$, $L$-Courant algebroids (contact Courant algebroids) are contact NQ-manifolds of degree $2$ \cite{grab}. There are some generalizations of Courant algebroids by relaxing the Jacobi identity of its bracket and they have suitable super-geometric interpretations. Pre-Courant algebroids are similar to Courant algebroids without the Jacobi identity of the bracket \cite{vais}. One the other hand, twisted Courant algebroids are similar to Courant algebroids on which the Jacobi identity of the bracket is controlled by a closed $4$-form \cite{hansen-strobl}. A twisted Courant algebroid corresponds to a $2$-term $Leib_\infty$-algebra \cite{sheng-liu}. One may also study similar generalizations of $L$-Courant algebroids, their supergeometric interpretations and associated algebraic structures.

\vspace{0.2cm}

\noindent {\bf Acknowledgements.} The author would like to thank Ms. Puja Mondal for carefully reading the manuscript and fix some typos. The research is supported by the Institute postdoctoral fellowship of IIT Kanpur (India).



\begin{thebibliography}{BFGM03}

\bibitem{abad-crainic} C. A. Abad and M. Crainic, Representations up to homotopy of Lie algebroids, {\em J. Reine Angew. Math.} 663 (2012), 91-126. 


\bibitem{chen-liu} Z. Chen and Z.-J. Liu, Omni-Lie algebroids, {\em J. Geom. Phys.}, 60 (2010), no. 5, 799-808.

\bibitem{chen-liu-sheng} Z. Chen, Z. Liu and Y. Sheng, $E$-Courant algebroids, {\em Int. Math. Res. Not. IMRN} 2010, no. 22, 4334-4376.

\bibitem{courant} T. J. Courant, Dirac manifolds, {\em Trans. Amer. Math. Soc.} 319 (1990), no. 2, 631-661.


\bibitem{das2} A. Das, Gauge transformations of Jacobi structures and contact groupoids, preprint, arXiv:1801.06681


\bibitem{grab} J. Grabowski, Graded contact manifolds and contact Courant algebroids, {\em J. Geom. Phys.}, 68 (2013), 27-58.






\bibitem{hansen-strobl} M. Hansen and T. Strobl, First class constrained systems and twisting of Courant algebroids by a closed 4-form, {\em Fundamental interactions}, 115-144, {\em World Sci. Publ., Hackensack, NJ}, 2010. 

\bibitem{lada-markl}
T. Lada and M. Markl, Strongly homotopy Lie algebras, {\it Comm. Algebra} 23 (1995), no. 6, 2147-2161.


\bibitem{liu-wein-xu} Z.-J. Liu, A. Weinstein and P. Xu, Manin triples for Lie bialgebroids, {\em J. Differential Geom.} 45 (1997), no. 3, 547-574.

\bibitem{marle} C. M. Marle, On Jacobi manifolds and Jacobi bundles, In: {\em Symplectic geometry, groupoids, and integrable systems (Berkeley, CA, 1989)}, 227-246, Math. Sci. Res. Inst. Publ., 20, {\em Springer, New York,} 1991. 

\bibitem{costa} J. M. Nunes da Costa and J. Clemente-Gallardo, Dirac structures for generalized Lie bialgebroids, {\em J. Phys. A} 37 (2004), no. 7, 2671-2692.

\bibitem{rogers1} C. L. Rogers, Courant algebroids from categorified symplectic geometry, preprint, arXiv:1001.0040

\bibitem{rogers} C. L. Rogers, $L_\infty$-algebras from multisymplectic geometry, {\em Lett. Math. Phys.} 100 (2012), no. 1, 29-50.

\bibitem{roy} D. Roytenberg, Courant algebroids, derived brackets and even symplectic supermanifolds, PhD Thesis, UC Berkeley (1999).


\bibitem{roy3} D. Roytenberg, AKSZ-BV formalism and Courant algebroid-induced topological field theories, {\em Lett. Math. Phys.} 79(2) (2007), 143-159.

\bibitem{roy-wein} D. Roytenberg and A. Weinstein, Courant algebroids and strongly homotopy Lie algebras, {\em Lett. Math. Phys.} 46 (1998), no. 1, 81-93.

\bibitem{severa} P. {\v S}evera, Letters to Alan Weinstein, http://sophia.dtp.fmph.uniba.sk/$\sim$severa/letters/(1998-1999).

\bibitem{severa-wein} P. {\v S}evera and A. Weinstein, Poisson Geometry with a $3$-Form Background, {\em Progr. Theoret. Phys. Suppl.} No. 144 (2001), 145-154.

\bibitem{sheng-liu} Y. Sheng and Z. Liu, Leibniz 2-algebras and twisted Courant algebroids, {\em Comm. Algebra} 41 (2013), no. 5, 1929-1953.

\bibitem{sheng-zhu} Y. Sheng and C. Zhu, Semidirect products of representations up to homotopy, {\em Pacific J. Math.} 249 (2011), no. 1, 211-236.

\bibitem{vais} I. Vaisman, Transitive Courant algebroids, {\em Int. J. Math. Sci.} 11 (2005), 1737-1758.


\bibitem{vitag2} L. Vitagliano, Dirac-Jacobi bundles, {\em J. Symplectic Geom.} 16 (2018), no. 2, 485-561.

\bibitem{wade} A. Wade, Conformal Dirac structures, {\em Lett. Math. Phys.} 53 (2000), no. 4, 331-348.

\bibitem{zambon} M. Zambon, $L_\infty$-algebras and higher analogues of Dirac structures and Courant algebroids, {\em J. Symplectic Geom.} 10 (2012), no. 4, 563-599.
\end{thebibliography}
\end{document}